\documentclass[reqno,tbtags]{amsproc}
\usepackage{amssymb}
\usepackage{epsfig}
\usepackage{chicago}
\usepackage{euscript}
\usepackage{eufrak}

\newtheorem{theorem}{Theorem}[section]
\newtheorem{corollary}{Corollary}[section]
\newtheorem{lemma}{Lemma}[section]

\theoremstyle{definition}

\newtheorem{case}{Case}[section]

\def\IME{\emph{In\-su\-ran\-ce: Ma\-the\-ma\-tics and Eco\-no\-mics\/}}

\newcommand{\expaN}[1]{\mathcal{E}_{#1}}

\newcommand{\Int}[2]{{\mathcal{#2}}_{#1}}
\newcommand{\KonstF}{C_{\mathcal{F}}\,}
\newcommand{\KonstS}{C_{\mathcal{S}}\,}
\def\ranV{X}
\def\a{a_Y}
\def\b{b_Y}
\def\d{a_T}
\def\g{b_T}
\def\aa{a}
\def\bb{b}
\def\paramY{\alpha}
\def\paramT{\beta}
\def\paramV{\theta}
\def\eqOK{=}

\def\cS{c^{*}}
\newcommand{\timeR}{\Upsilon}
\renewcommand{\P}{\mathsf{P}}

\newcommand{\D}{\mathsf{D}}
\newcommand{\E}{\mathsf{E}}
\newcommand{\Y}[1]{Y_{#1}}
\newcommand{\T}[1]{T_{#1}}
\newcommand{\Bessel}[1]{I_{#1}}
\newcommand{\UGauss}[2]{\varPhi_{\left({#1},{#2}\right)}}
\newcommand{\Ugauss}[2]{\varphi_{\left({#1},{#2}\right)}}
\newcommand{\homN}[1]{N_{#1}}
\newcommand{\homV}[1]{V_{#1}}

\numberwithin{equation}{section}
\begin{document}
\author[Vsevolod K. Malinovskii]{Vsevolod K. Malinovskii}

\author[Konstantin V. Malinovskii]{Konstantin V. Malinovskii}


\address{Central Economics and Mathematics Institute (CEMI) of Russian Academy of Science,
117418, Nakhimovskiy prosp., 47, Moscow, Russia and Gubkin Russian State
University of Oil and Gas, 119991, Moscow, GSP-1, Leninsky prosp., 65.}

\email{malinov@orc.ru, kmalinovskii@mail.ru}

\title[ON APPROXIMATIONS FOR THE DISTRIBUTION OF FIRST LEVEL CROSSING TIME]{ON APPROXIMATIONS
FOR THE DISTRIBUTION OF FIRST LEVEL CROSSING TIME}

\maketitle

\begin{abstract}
We investigate performance of approximations put forth in \citeNP{[Malinovskii
2017a]} and \citeNP{[Malinovskii 2017b]} for the distribution of the time of
first level $u$ crossing by the random process $\homV{s}-cs$, $s>0$, where
$\homV{s}$ is compound renewal process. In the case of Exponential
inter-renewal and jump size random variables, we compare the approximations
with exact and with simulation results. In a few other cases including Erlang
and Pareto inter-renewal and jump size random variables, where exact results
are absent, we compare the approximations with simulation results.
\end{abstract}

\section{Introduction}\label{edtfyjuuku}

Let random variables $\T{1}$, i.i.d. $\T{i}\overset{d}{=}T$, $i=2,3,\dots$,
i.i.d. $\Y{i}\overset{d}{=}Y$, $i=1,2,\dots$, be positive, and all mutually
independent. Compound renewal process with time $s\geqslant 0$ is
\begin{equation*}
\homV{s}=\sum_{i=1}^{\homN{s}}\Y{i},
\end{equation*}
or $0$, if $\homN{s}=0$ (or $\T{1}>s$), where
$\homN{s}=\max\left\{n>0:\sum_{i=1}^{n}\T{i}\leqslant s\right\}$, or $0$, if
$\T{1}>s$. Within the renewal model, $\T{1}$ is called interval between
starting time zero and time of the first renewal, $\T{i}$, $i=2,3,\dots$ are
called inter-renewal times, and $\Y{i}$ are called sizes of jumps at the
moments of renewals; the distribution of $\T{1}$ may be different from the
distribution of the other interclaim intervals, i.e., from the distribution of
$T$.

By $f_{\T{1}}(x)$ and $f_{T}(x)$ we denote p.d.f. of random variables $\T{1}$
and $T$. By $f_{Y}(x)$ we denote p.d.f. of random variable $Y$. Throughout the
entire presentation, p.d.f. $f_{T}(x)$ and $f_{Y}(x)$ are assumed bounded from
above by a finite constant. By $F_{T}(x)$ and $F_{Y}(x)$ we denote respective
c.d.f.

For $c>0$ and $u>0$, the random variable
\begin{equation*}
\timeR=\inf\left\{s>0:\homV{s}-cs>u\right\},
\end{equation*}
or $+\infty$, as $\homV{s}-cs\leqslant u$ for all $s>0$, is called the time of
the first level $u$ crossing by the process $\homV{s}-cs$. It is easily seen
that for $t>0$
\begin{equation*}
\P\{\timeR\leqslant
t\}=\int_{0}^{t}\P\{u+cv-\Y{1}<0\}f_{\T{1}}(v)dv+\int_{0}^{t}\P\{v<\timeR\leqslant
t\mid\T{1}=v\}f_{\T{1}}(v)dv.
\end{equation*}

The distribution of $\timeR$ appears in many branches of applied probability,
including risk and queueing theories, and was considered by many authors. For
it, there are many closed-form formulas and approximations, derived by
different techniques.

Investigating $\P\{\timeR\leqslant t\}$, we will be focused on
$\P\{v<\timeR\leqslant t\mid\T{1}=v\}$. As soon as the distribution of $T_1$ is
specified, the former is straightforward from the latter. For example, for
$\T{1}$ Exponential with parameter $\paramT$, we have
\begin{equation*}
\P\{\timeR\leqslant t\}=\int_{0}^{t}\P\{\Y{1}>u+cv\}e^{-\paramT
v}dv+\int_{0}^{t}\P\{v<\timeR\leqslant t\mid\T{1}=v\}e^{-\paramT v}dv.
\end{equation*}
If, in addition, $Y$ is, e.g., Exponential with parameter $\paramY$, we have
\begin{multline*}
\P\{\timeR\leqslant t\}=\lambda\int_{0}^{t}e^{-\mu(u+cv)}e^{-\lambda
v}dv+\lambda\int_{0}^{t}\P\{v<\timeR\leqslant t\mid\T{1}=v\}e^{-\lambda v}dv
\\
=\frac{\lambda e^{-\mu
u}}{\lambda+c\mu}\big(1-e^{-(\lambda+c\mu)t}\big)+\lambda\int_{0}^{t}
\P\{v<\timeR\leqslant t\mid\T{1}=v\}e^{-\lambda v}dv.
\end{multline*}

The goal of this paper is to get an idea of the quality of the approximations
for $\P\{v<\timeR\leqslant t\mid\T{1}=v\}$ obtained in \citeNP{[Malinovskii
2017a]} and \citeNP{[Malinovskii 2017b]}, in which inverse Gaussian and
generalized inverse Gaussian distributions are involved, and which seem to be
new. Remarkable is that they are derived under a set of conditions similar to
those usually imposed in the common local central limit theorem. For this
purpose, these approximations will be compared with exact formulas available in
the Exponential case, and with numerical simulation results obtained in a few
other cases, including Mixture of two Exponentials, Erlang and Pareto.

This paper is organized as follows. In Section \ref{sdtguyrtikt}, we present
two approximations obtained in \citeNP{[Malinovskii 2017a]} and
\citeNP{[Malinovskii 2017b]}. In Section \ref{wertg54ryher}, for $T$ and $Y$
Exponential, we compare approximations with exact result given in
Theorem~\ref{dthyjrt}. In Section \ref{ewrtherherhr}, we deal with performance
of these approximations in a few cases when $T$ and $Y$ are non-Exponential. As
a benchmark in all these non-Exponential cases, we use numerical simulation
results. In Section \ref{srdtyrujhrt}, we make several conclusive remarks.

\section{Approximations}\label{sdtguyrtikt}

Put\footnote{Here $\D{Y}=\E(Y-\E{Y})^2$ and $\D{T}=\E(T-\E{T})^2$.}
$M={\E{T}}/{\E{Y}}$, $D^2=((\E{T})^2\D{Y}+(\E{Y})^2\D{T})/(\E{Y})^3$, write
$\Ugauss{m}{s^2}$ for p.d.f. of a normal distribution with mean $m$ and
variance $s^2$, and for $c>0$, $u>0$, $t>0$, $0<v<t$ introduce
\begin{equation*}
\begin{aligned}
\Int{t}{M}(u,c,v)&=\int_{0}^{\frac{c(t-v)}{u+cv}}\frac{1}{1+x}
\,\Ugauss{cM(1+x)}{\frac{c^2D^2(1+x)}{u+cv}}(x)dx
\\[4pt]
&\eqOK
\bigg[\UGauss{0}{1}\Big(\tfrac{\sqrt{u+cv}}{cD\sqrt{x}}\big(x(1-cM)-1\big)\Big)
\\[-4pt]
&\hskip
28pt+\exp\Big\{2\frac{u+cv}{c^2D^2}(1-cM)\Big\}\UGauss{0}{1}\Big(-\tfrac{\sqrt{u+cv}}{cD
\sqrt{x}}\big(x(1-cM)+1\big)\Big)\bigg]\bigg|_{x=1}^{\frac{c(t-v)}{u+cv}+1}
\end{aligned}
\end{equation*}
and
\begin{equation}\label{wsdrthyjkr}
\expaN{t}(u,c,v)=\Int{t}{M}(u,c,v)+\KonstF\Int{t}{F}(u,c,v)+\KonstS\Int{t}{S}(u,c,v),
\end{equation}
where
\begin{equation*}
\begin{aligned}
\Int{t}{F}(u,c,v)&=\int_{0}^{\frac{c(t-v)}{u+cv}}\frac{x-Mc(1+x)}{(1+x)^2}
\Ugauss{cM(1+x)}{\frac{c^2D^2(1+x)}{u+cv}}(x)dx,
\\
\Int{t}{S}(u,c,v)&=\frac{u+cv}{c^2D^2}\int_{0}^{\frac{c(t-v)}{u+cv}}\frac{(x-Mc(1+x))^3}{(1+x)^3}
\Ugauss{cM(1+x)}{\frac{c^2D^2(1+x)}{u+cv}}(x)dx,
\end{aligned}
\end{equation*}
and
\begin{equation*}
\begin{aligned}
\KonstF&\eqOK\frac{\E(T-\E{T})^3}{2cD^2\D{T}}\bigg(\dfrac{(\E{T})^2\D{Y}}{D^2(\E{Y})^3}-1\bigg)-\frac{\E{T}\E(Y-\E{Y})^3}{2cD^2\E{Y}\D{Y}}
\bigg(\dfrac{\D{T}}{D^2\E{Y}}-1\bigg)+\frac{\E{T}}{2cD^2},
\\[8pt]
\KonstS&\eqOK\frac{\E(T-\E{T})^3}{6cD^4\E{Y}}
-\dfrac{(\E{T})^3\E(Y-\E{Y})^3}{6cD^4(\E{Y})^4}
+\frac{\E{T}\D{Y}}{2cD^2(\E{Y})^2}.
\end{aligned}
\end{equation*}
Recall that (see Theorem 4.2 in \citeNP{[Malinovskii 2017b]})
\begin{equation*}
\begin{aligned}
\Int{t}{F}(u,c,v)&\eqOK-\frac{c^2D^2}{u+cv}\bigg[\UGauss{0}{1}\Big(\tfrac{\sqrt{u+cv}}{cD\sqrt{x}}
\big(x(1-cM)-1\big)\Big)
\\[-2pt]
&\hskip 32pt+\exp\Big\{2\frac{u+cv}{c^2D^2}(1-cM)\Big\}
\UGauss{0}{1}\Big(-\tfrac{\sqrt{u+cv}}{cD\sqrt{x}}\big(x(1-cM)+1\big)\Big)
\bigg]\bigg|_{x=1}^{\frac{c(t-v)}{u+cv}+1}
\\
&+2(1-cM)\exp\Big\{2\frac{u+cv}{c^2D^2}(1-cM)\Big\}\UGauss{0}{1}\Big(-\tfrac{\sqrt{u+cv}}{cD\sqrt{x}}
\big(x(1-cM)+1\big)\Big)\bigg|_{x=1}^{\frac{c(t-v)}{u+cv}+1}
\\
&-\frac{2cD}{\sqrt{2\pi x(u+cv)}}\exp\Big\{-\frac{u+cv}{2xc^2D^2}
\big(x(1-cM)-1\big)^2\Big\}\bigg|_{x=1}^{\frac{c(t-v)}{u+cv}+1},
\end{aligned}
\end{equation*}
and that (see Theorem 4.3 in \citeNP{[Malinovskii 2017b]})
\begin{equation*}
\begin{aligned}
\Int{t}{S}(u,c,v)&\eqOK-\frac{3\,c^2D^2}{u+cv}\bigg[\UGauss{0}{1}\Big(\tfrac{\sqrt{u+cv}}{cD\sqrt{x}}
\big(x(1-cM)-1\big)\Big)
\\[-4pt]
&\hskip 28pt+\exp\Big\{2\frac{u+cv}{c^2D^2}(1-cM)\Big\}
\UGauss{0}{1}\Big(-\tfrac{\sqrt{u+cv}}{cD\sqrt{x}}\big(x(1-cM)+1\big)\Big)\bigg]
\bigg|_{x=1}^{\frac{c(t-v)}{u+cv}+1}
\\
&+2(1-cM)\Big(3-4\frac{u+cv}{c^2D^2}(1-cM)\Big)
\\[-4pt]
&\hskip
28pt\times\exp\Big\{2\frac{u+cv}{c^2D^2}(1-cM)\Big\}\UGauss{0}{1}\Big(-\tfrac{\sqrt{u+cv}}{cD\sqrt{x}}
\big(x(1-cM)+1\big)\Big)\bigg|_{x=1}^{\frac{c(t-v)}{u+cv}+1}
\\
&-\frac{\sqrt{2}\,cD}{\sqrt{\pi}\sqrt{u+cv}\,x^{3/2}}\Big(3
\Big(1-\frac{u+cv}{c^2D^2}(1-cM)\Big)x+\frac{u+cv}{c^2D^2}\Big)
\\[-4pt]
&\hskip 28pt\times\exp\Big\{-\frac{u+cv}{2xc^2D^2} \big(x(1-cM)-1\big)^2\Big\}
\bigg|_{x=1}^{\frac{c(t-v)}{u+cv}+1}.
\end{aligned}
\end{equation*}

Let us formulate two core results of \citeNP{[Malinovskii 2017a]} and
\citeNP{[Malinovskii 2017b]}.

\begin{theorem}[\citeNP{[Malinovskii 2017a]}]\label{srdthjrf}
In the above model, let p.d.f. $f_{T}(y)$ and $f_{Y}(y)$ be bounded from above
by a finite constant, $D^2>0$, $\E({T}^3)<\infty$, $\E({Y}^3)<\infty$. Then for
$c>0$, for fixed $0<v<t$ we have
\begin{equation*}
\sup_{t>v}\Big|\,\P\{v<\timeR\leqslant t\mid\T{1}=v\}-\Int{t}{M}(u,c,v)\Big|
=\underline{O}\bigg(\frac{\ln(u+cv)}{u+cv}\bigg),
\end{equation*}
as $u+cv\to\infty$.
\end{theorem}

\begin{theorem}[\citeNP{[Malinovskii 2017b]}]\label{erty5uyh4tj}
In the above model, let p.d.f. $f_{T}(y)$ and $f_{Y}(y)$ be bounded from above
by a finite constant, $D^2>0$, $\E({T}^4)<\infty$, $\E({Y}^4)<\infty$. Then for
$c>0$, for fixed $0<v<t$ we have
\begin{equation*}
\sup_{t>v}\Big|\,\P\{v<\timeR\leqslant t\mid\T{1}=v\}-\expaN{t}(u,c,v)\Big|
=\underline{O}\bigg(\frac{\ln(u+cv)}{(u+cv)^2}\bigg),
\end{equation*}
as $u+cv\to\infty$.
\end{theorem}

Bearing in mind that $\Int{t}{F}(u,c,v)$ and $\Int{t}{S}(u,c,v)$ both are
$\underline{O}((u+cv)^{-1})$, as $u+cv\to\infty$, Theorem~\ref{erty5uyh4tj} is
a development of Theorem~\ref{srdthjrf} that may be called asymptotic expansion
with the first correction term written down explicitly. This theoretical
advancement is a kind of results commonly known for many limit theorems of the
theory of probability.

However, these two approximations viewed merely as two different tools
available for numerical evaluation of $\P\{v<\timeR\leqslant t\mid\T{1}=v\}$
need not be one strictly better than the other\footnote{It is not surprising
that in some cases corrected approximation $\expaN{t}(u,c,v)$ may yield less
accurate result than the main term approximation $\Int{t}{M}(u,c,v)$; in some
cases the former, in contrast to the latter, may assume negative values.}
uniformly in all sets of fixed parameters and variables. For better
understanding of how these tools work, it requires further insight into
performance of these approximations.

It is noteworthy that both Theorems~\ref{srdthjrf} and \ref{erty5uyh4tj} are in
a sense unready for effective analytical evaluation of the accuracy of the
approximations proposed there because the right hand sides are given in terms
of $\underline{O}(\cdot)$; it does not allow us to assess effectively the
impact of the distribution of $T$ and $Y$ on the actual performance. It could
be done if the estimates in terms of $\underline{O}(\cdot)$ were replaced by
computable upper bounds, with constants explicitly written in terms of, e.g.,
cumulants of $T$ and $Y$.

It appears that this development of Theorems~\ref{srdthjrf} and
\ref{erty5uyh4tj} can be done by a further development of the methods suggested
in \citeNP{[Malinovskii 2017a]} and \citeNP{[Malinovskii 2017b]}, but it seems
to be an extremely laborious analytical work. So, in
Sections~\ref{wertg54ryher} and \ref{ewrtherherhr} we will be focused on
$\Int{t}{M}(u,c,v)$ and $\expaN{t}(u,c,v)$ in a few test cases to get a
numerical insight into performance of these approximations.

\section{Performance of approximations 
when $T$ and $Y$ are Exponential}\label{wertg54ryher}

\begin{figure}[t]
\includegraphics[scale=0.9]{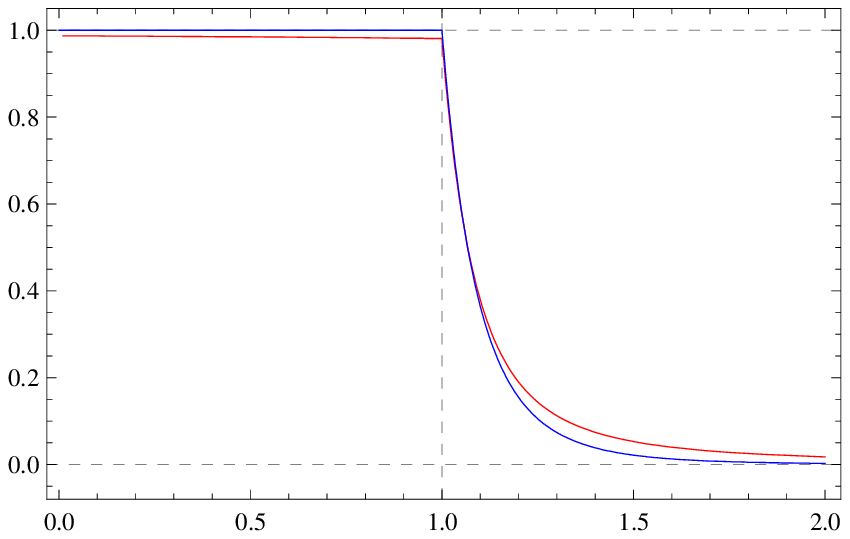}
\includegraphics[scale=0.9]{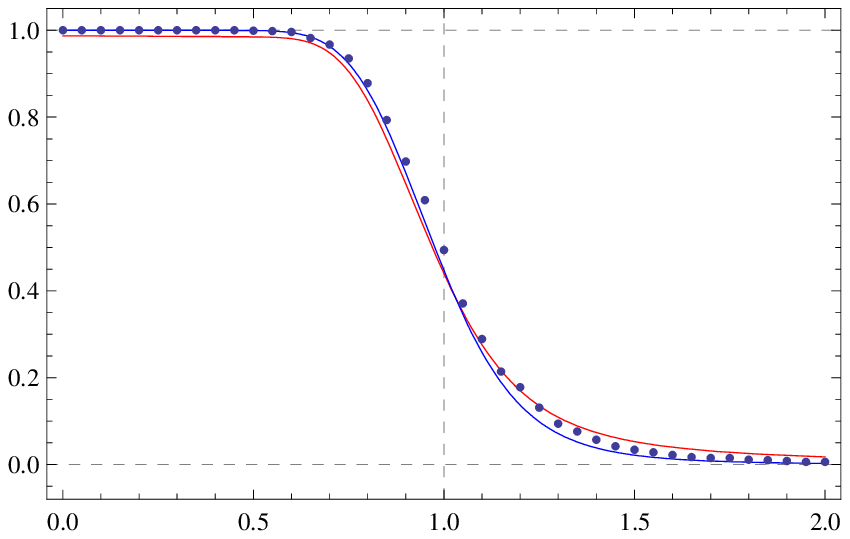}
\caption{\small\!(\emph{Exponential~$Y$--Exponential~$T$})\, Graphs ($X$-axis
is $c$) of the functions $\P\{v<\timeR\leqslant t\mid \T{1}=v\}$ (blue) given
in Theorem~\ref{dthyjrt} and $\Int{t}{M}(u,c,v)$ (red) for Exponential $T$ with
parameter $\paramT$ and Exponential $Y$ with parameter $\paramY$, as
$\paramT=\paramY=1$, $v=0$, $u=10$, $t=\infty$ (above), $t=100$ (below). By
dots, shown are the results of simulation ($\Delta c=0.05$, $N=1000$) according
to the algorithm described in Section~\ref{ewrtyurtujtr}.}\label{sdrtfrjmty}
\end{figure}


\begin{figure}[t]
\includegraphics[scale=0.9]{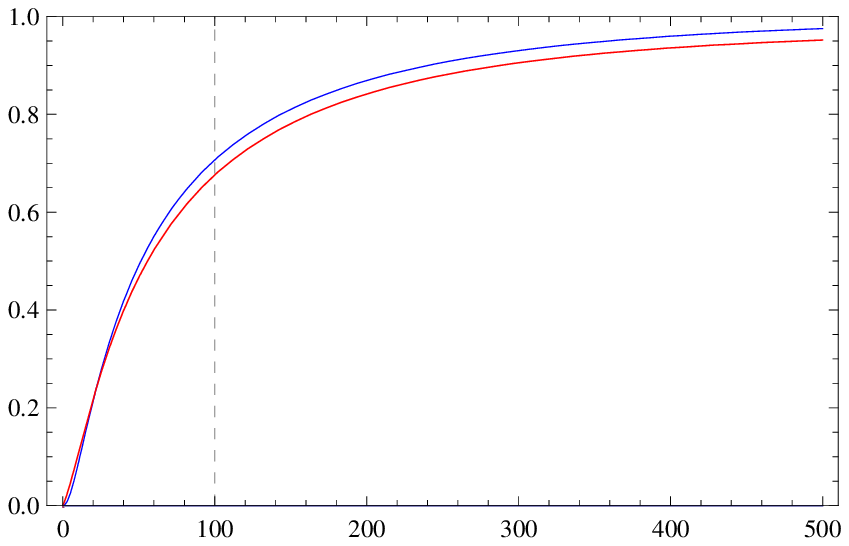}
\includegraphics[scale=0.9]{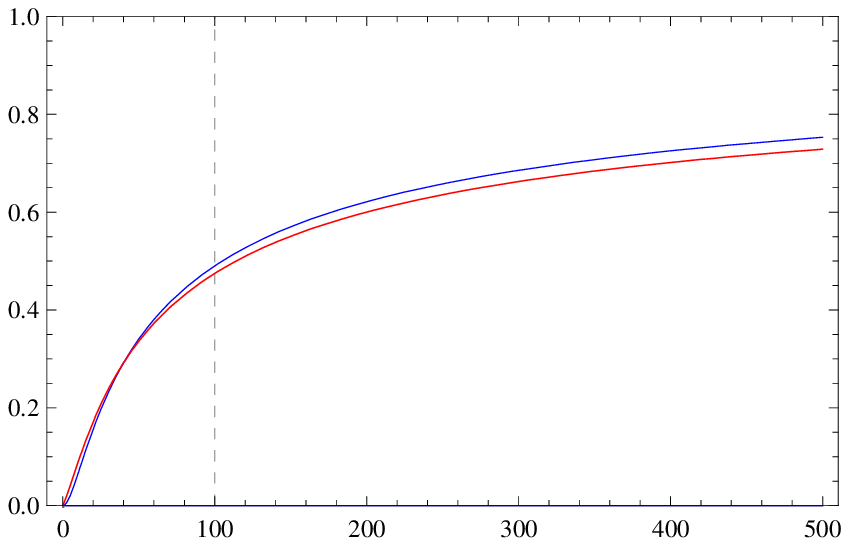}
\includegraphics[scale=0.9]{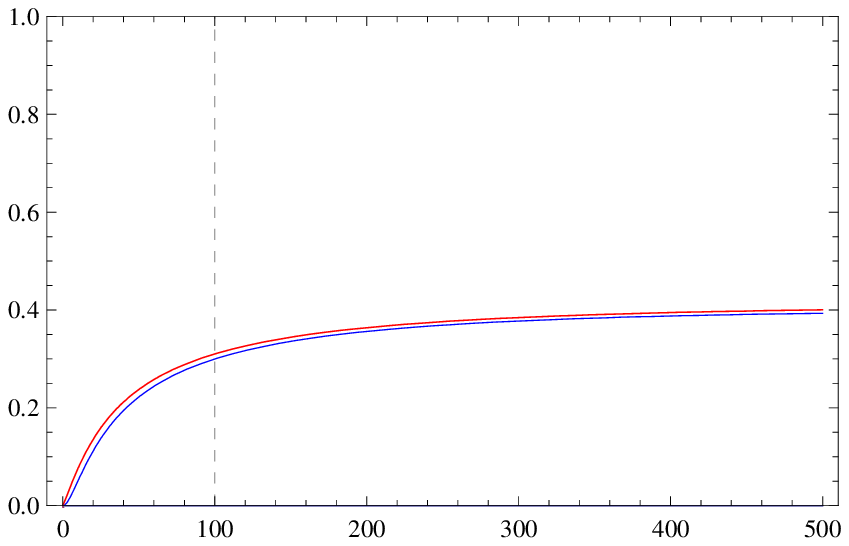}
\caption{\small\!(\emph{Exponential~$Y$--Exponential~$T$})\, Graphs ($X$-axis
is $t$) of the functions $\P\{v<\timeR\leqslant t\mid \T{1}=v\}$ (blue) given
in Theorem~\ref{dthyjrt} and $\Int{t}{M}(u,c,v)$ (red) for Exponential $T$ with
parameter $\paramT$ and Exponential $Y$ with parameter $\paramY$, as
$\paramT=\paramY=1$, $v=0$, $u=10$, $c=0.9$ (above), $c=1$ (middle), $c=1.1$
(below).}\label{drgfhsdc}
\end{figure}

\begin{figure}[t]
\includegraphics[scale=0.9]{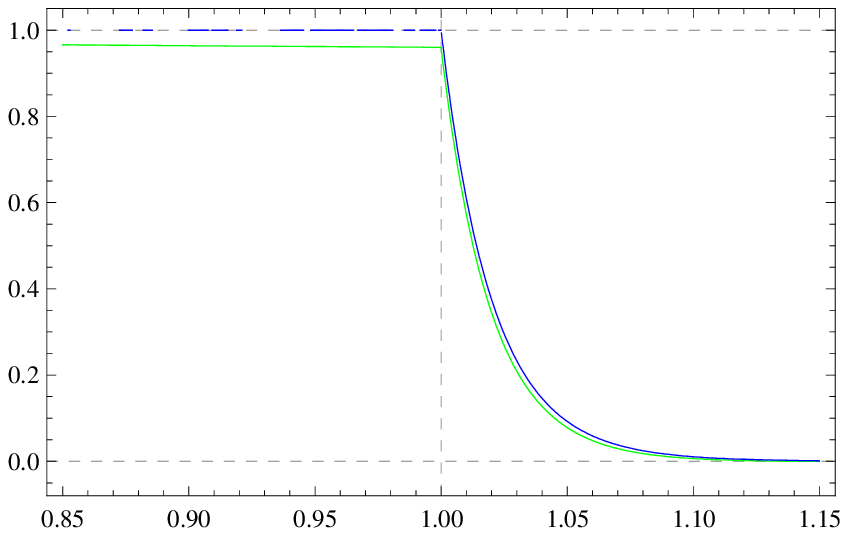}
\includegraphics[scale=0.9]{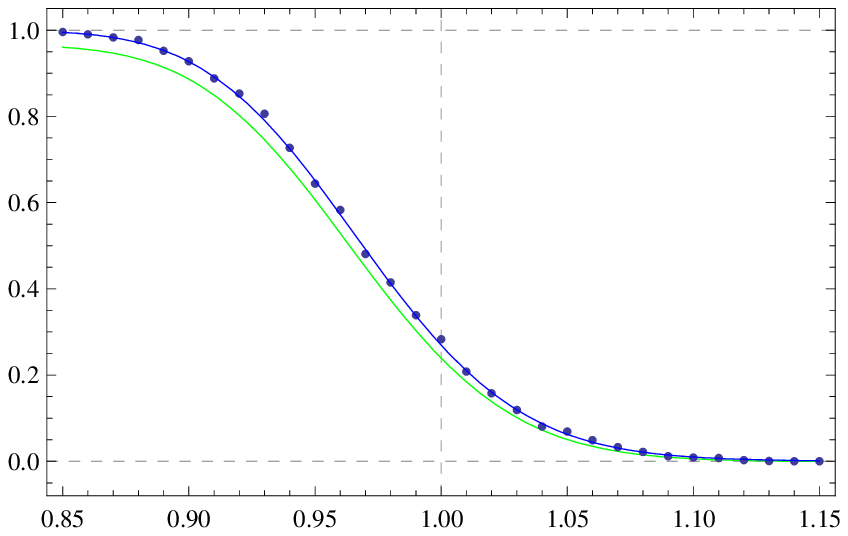}
\caption{\small\!(\emph{Exponential~$Y$--Exponential~$T$})\, Graphs ($X$-axis
is $c$) of the functions $\P\{v<\timeR\leqslant t\mid \T{1}=v\}$ (blue) given
in Theorem~\ref{dthyjrt} and $\expaN{t}(u,c,v)$ (green) defined in equation
\eqref{wsdrthyjkr}, for Exponential $T$ with parameter $\paramT$ and
Exponential $Y$ with parameter $\paramY$, as $\paramT=\paramY=1$, $v=0$,
$u=50$, $t=\infty$ (above), $t=1000$ (below). By dots, shown are the results of
simulation ($\Delta c=0.01$, $N=1000$) according to the algorithm described in
Section~\ref{ewrtyurtujtr}.}\label{ertgerhrh}
\end{figure}


\begin{figure}[t]
\includegraphics[scale=0.9]{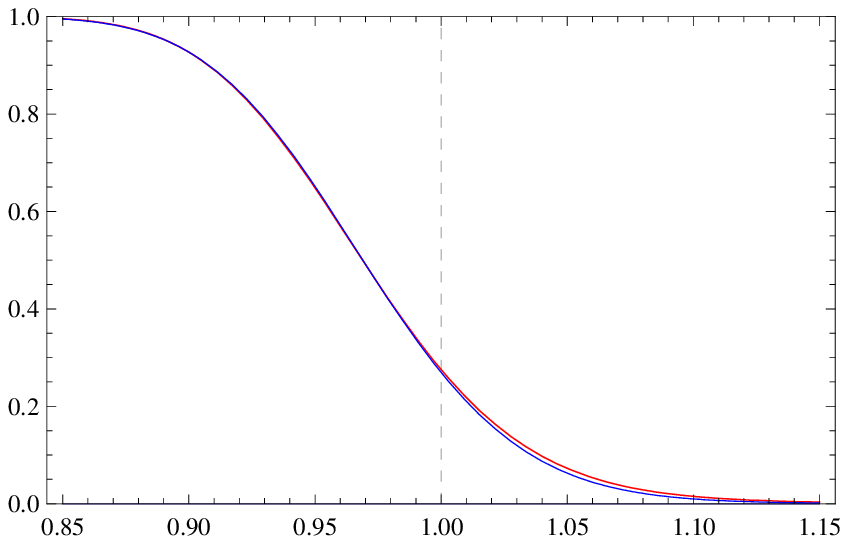}
\includegraphics[scale=0.9]{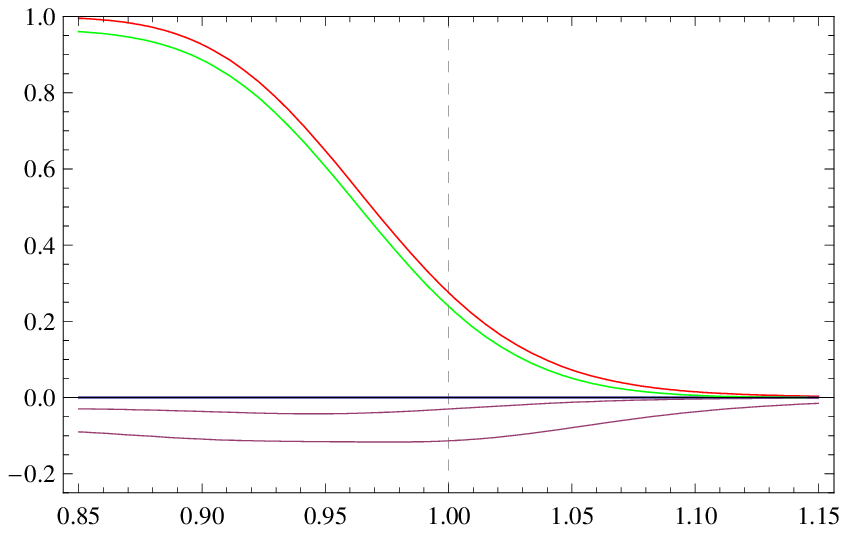}
\includegraphics[scale=0.9]{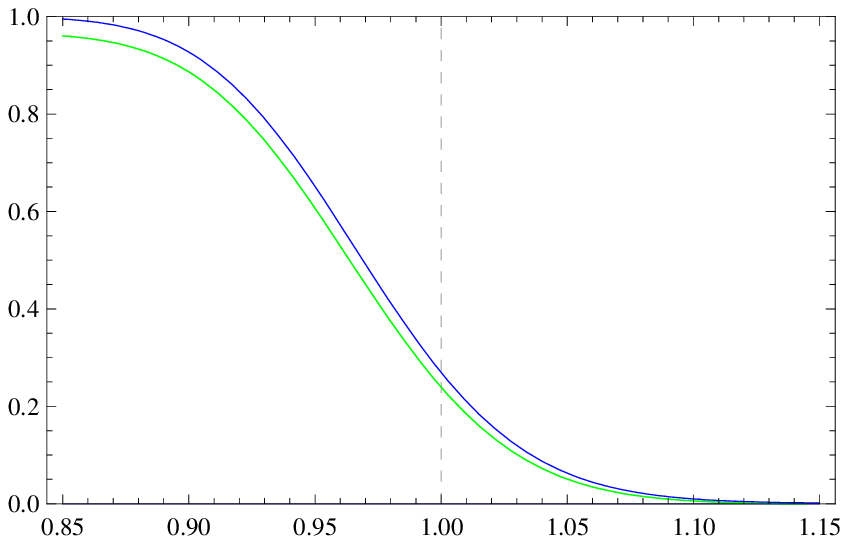}
\caption{\small\!(\emph{Exponential~$Y$--Exponential~$T$})\, Above: graphs
($X$-axis is $c$) of the functions $\P\{v<\timeR\leqslant t\mid \T{1}=v\}$
(blue) given in Theorem~\ref{dthyjrt} and $\Int{t}{M}(u,c,v)$ (red). Middle:
graphs of $\Int{t}{M}(u,c,v)$ (red), $\expaN{t}(u,c,v)$ (green),
$\Int{t}{F}(u,c,v)$, and $\Int{t}{S}(u,c,v)$. Below: graphs of
$\P\{v<\timeR\leqslant t\mid\T{1}=v\}$ (blue) and $\expaN{t}(u,c,v)$ (green).
Here $\paramT=\paramY=1$, $t=1000$, $v=0$, $u=50$. In this case, $\cS=1/M=1$,
$D^2=2$, $\KonstF=0.25\,c^{-1}$, $\KonstF=0.25\,c^{-1}$.}\label{sdfghgfmkgh}
\end{figure}

Let us consider the case when $T$ is Exponential with parameter $\paramT>0$ and
$Y$ is Exponential with parameter $\paramY>0$. Plainly,
\begin{equation*}
\E{T}=\frac{1}{\paramT},\quad\D{T}=\frac{1}{\paramT^2},\quad\E{Y}=\frac{1}{\paramY},\quad\D{Y}=\frac{1}{\paramY^2},
\end{equation*}
and
\begin{equation*}
M=\frac{\paramY}{\paramT},\quad D^2=\frac{2\paramY}{\paramT^2}.
\end{equation*}
It is calculated straightforwardly that
\begin{equation*}
\KonstF\eqOK\frac{\paramT}{4\paramY\,c},\quad
\KonstS\eqOK\frac{\paramT}{4\paramY\,c}.
\end{equation*}

Denote by $\Bessel{\nu}(z)$ the modified Bessel function of the first kind of
order $\nu$ (see, e.g., \citeNP{[Abramowitz Stegun 1972]}). For $T$ and $Y$
Exponential, one can get an exact formula for $\P\{v<\timeR\leqslant t\mid
\T{1}=v\}$ as follows.

\begin{theorem}\label{dthyjrt}
Assuming that $T$ and $Y$ are Exponential with parameters $\paramT>0$ and
$\paramY>0$ respectively, for $0<v<t$ we have
\begin{multline*}
\P\{v<\timeR\leqslant t\mid\T{1}=v\}=\sqrt{\paramY\paramT
c}\,(v+u/c)e^{-\paramY u}e^{-\paramY cv}
\\
\times\int_{0}^{t-v}\frac{\Bessel{1}(2\sqrt{\paramY\paramT
c(y+v+u/c)y})}{\sqrt{(y+v+u/c)y}} e^{-(\paramY c+\paramT)y}dy.
\end{multline*}
\end{theorem}

\begin{proof}[Proof of Theorem~\ref{dthyjrt}]
For $M(s)=\inf\big\{k\geqslant 1:\sum_{i=1}^{k}Y_{i}>s\big\}-1$ and $0<v<t$, we
have (see equation (2.1) in \citeNP{[Malinovskii 2017a]})
\begin{equation}\label{qwretgjhmg}
\P\{v<\timeR\leqslant
t\mid\T{1}=v\}=\int_{v}^{t}\dfrac{u+cv}{u+cz}\sum_{n=1}^{\infty}
\P\{M(u+cz)=n\}f_{T}^{*n}(z-v)dz.
\end{equation}

For $T$ Exponential with parameter $\paramT$, we have
\begin{equation}\label{sadfgrfhber}
f_{T}^{*n}(z-v)=\paramT\frac{(\paramT (z-v))^{n-1}}{(n-1)!}e^{-\paramT
(z-v)},\quad n=1,2,\dots.
\end{equation}

For $Y$ Exponential with parameter $\paramY$, we have
\begin{equation}\label{saerfgewgh}
\P\{M(u+cz)=n\}=\frac{(\paramY(u+cz))^n}{n!}e^{-\paramY(u+cz)},\quad
n=1,2,\dots.
\end{equation}

Bearing in mind that modified Bessel function of the first kind of order $1$
is\footnote{See e.g. \citeNP{[Abramowitz Stegun 1972]}, or \citeNP{[Watson
1945]}, or Chapter XVII, Section 17.7 in \citeNP{[Whittaker Watson 1963]}.}
\begin{equation*}
\Bessel{1}(z)=\sum_{k=0}^{\infty}\frac{1}{k!\,(k+1)!}\left(\frac{z}{2}\right)^{2k+1}
=\sum_{n=1}^{\infty}\frac{1}{n!\,(n-1)!}\left(\frac{z}{2}\right)^{2n-1},
\end{equation*}
we put \eqref{sadfgrfhber} and \eqref{saerfgewgh} in \eqref{qwretgjhmg}. We
have
\begin{multline*}
\int_{v}^{t}\frac{u+cv}{u+cz}e^{-\paramY(u+cz)}\sum_{n=1}^{\infty}\frac{(\paramY(u+cz))^n}{n!}
\paramT\frac{(\paramT (z-v))^{n-1}}{(n-1)!}e^{-\paramT (z-v)}dz
\\
=\paramY\paramT\int_{v}^{t}(u+cv)
\sum_{n=1}^{\infty}\frac{\paramY^{n-1}\paramT^{n-1}(u+cz)^{n-1}(z-v)^{n-1}}{n!(n-1)!}e^{-\paramY(u+cz)}e^{-\paramT
(z-v)}dz
\\
=\sqrt{\paramY\paramT c}\,(v+u/c)e^{-\paramY u}e^{-\paramY cv}
\int_{0}^{t-v}\frac{\Bessel{1}(2\sqrt{\paramY\paramT
c(y+v+u/c)y})}{\sqrt{(y+v+u/c)y}} e^{-(\paramY c+\paramT)y}dy,
\end{multline*}
as required. In the last equation we made the change of variables: $z-v=y$.
\end{proof}

In Figs.~\ref{sdrtfrjmty}--\ref{sdfghgfmkgh}, performance of the approximations
$\Int{t}{M}(u,c,v)$, $\expaN{t}(u,c,v)$ of Theorems~\ref{srdthjrf}
and~\ref{erty5uyh4tj} is visualized, when $T$ is Exponential with parameter
$\paramT=1$ and $Y$ is Exponential with parameter $\paramY=1$.

In Figs.~\ref{sdrtfrjmty} and~\ref{ertgerhrh} respectively, we compare
$\Int{t}{M}(u,c,v)$ and $\expaN{t}(u,c,v)$ with the exact numerical values for
$\P\{v<\timeR\leqslant t\mid \T{1}=v\}$ given in Theorem~\ref{dthyjrt}, as $c$
varies and $t$ is fixed\footnote{Note that $u$ and $t$ in
Figs.~\ref{sdrtfrjmty} and~\ref{ertgerhrh} are set different.}. The results of
simulation according to the algorithm described in Section~\ref{ewrtyurtujtr},
are shown by dots. In Fig.~\ref{drgfhsdc}, it is done, as $t$ varies and $c$ is
fixed. Figs.~\ref{sdrtfrjmty}--\ref{ertgerhrh} demonstrate good accuracy of
both approximations $\Int{t}{M}(u,c,v)$ and $\expaN{t}(u,c,v)$ throughout all
chosen range of $c$ and $t$.

In Fig.~\ref{sdfghgfmkgh}, visualized is performance of the approximation
$\Int{t}{M}(u,c,v)$ of Theorem~\ref{srdthjrf} (above), and of the approximation
$\expaN{t}(u,c,v)$ of Theorem~\ref{erty5uyh4tj} (below); both are compared with
the exact numerical values for $\P\{v<\timeR\leqslant t\mid \T{1}=v\}$ given in
Theorem~\ref{dthyjrt}.  On the one hand, it may be seen that in this test case
the former lies closer to the exact than the latter\footnote{It is no surprise,
bearing in mind in particular that $u$ is taken equal to $50$, i.e. is rather
moderate. This and the following remark is just a curious observation rather
than a characteristic feature.}. On the other hand, the latter lies
consistently below the exact all over $c$ in all range chosen, while the latter
does not. Visualization is also done (middle) for the components
$\Int{t}{F}(u,c,v)$ and $\Int{t}{S}(u,c,v)$ which, together with the factors
$\KonstF$ and $\KonstS$, produce $\expaN{t}(u,c,v)$ as in \eqref{wsdrthyjkr}.
It is noteworthy that lower Fig.~\ref{ertgerhrh} and lower
Fig.~\ref{sdfghgfmkgh} are the same, except that simulation results are shown
on the former.

\section{Performance of approximations 
when $T$ and $Y$ are non-Exponential}\label{ewrtherherhr}

First, we recall the properties of three non-Exponential distributions which
will be selected as distributions of $T$ and $Y$. Second, we describe the
algorithm of numerical simulation. Finally, we present the main results of this
section.

\subsection{Three non-Exponential distributions}\label{wdrthrejher}

Let us select the Mixture of two Exponentials and Erlang distributions as two
non-Exponential distributions which properties strongly resemble those of
Exponential, and Pareto, which properties are far from those of Exponential; in
particular, it is well known that Pareto is heavy-tailed.

\begin{case}[Mixture of two Exponentials]\label{serfgerwghe}
The random variable $\ranV$ is a Mixture of two Exponentials if for
$0<\lambda_1<\lambda_2<\infty$ and for $p$, $q$ such that $p+q=1$, $0\leqslant
p,q\leqslant 1$, its p.d.f. is
\begin{equation*}
f_{\ranV}(x)=\lambda_1pe^{-\lambda_1x}+\lambda_2qe^{-\lambda_2 x},\quad x>0,
\end{equation*}
and $0$, as $x\leqslant 0$. Plainly, the corresponding c.d.f.
$F_{\ranV}(x)=\int_{0}^{x}f_{\ranV}(z)dz$ is
\begin{equation}\label{ewrherhjedtj}
F_{\ranV}(x)=1-(pe^{-\lambda_1x}+qe^{-\lambda_2 x}),\quad x>0,
\end{equation}
and $0$, as $x\leqslant 0$.
By direct calculations, it is easy to check that
\begin{equation*}
\begin{gathered}
\E{\ranV}=\frac{p}{\lambda_1}+\frac{q}{\lambda_2},\quad
\D{\ranV}=\frac{q\lambda_1^2+p\lambda_2^2+pq(\lambda_1-\lambda_2)^2}{\lambda_1^2\lambda_2^2},\quad
\\[6pt]
\E(\ranV-\E{\ranV})^3=-\frac{6pq^2}{\lambda_1^2
\lambda_2}-\frac{6p^2q}{\lambda_1
\lambda_2^2}+\frac{2p\big(3q+p^2\big)}{\lambda_1^3}+\frac{2q\big(3p+q^2\big)}{\lambda_2^3}.
\end{gathered}
\end{equation*}
\end{case}

\begin{case}[Erlang]\label{wergthrt}
The random variable $\ranV$ is Erlang if for $\paramV>0$ and integer $k$ its
p.d.f. is
\begin{equation*}
f_{\ranV}(x)=\frac{\paramV^k x^{k-1}}{\Gamma(k)}e^{-x\paramV},\quad x>0,
\end{equation*}
and $0$, as $x\leqslant 0$.
It is well known that Erlang is a particular case of the Gamma p.d.f.

By direct calculations, it is easy to check that
\begin{equation*}
\E{\ranV}=\frac{k}{\paramV},\quad\D{\ranV}=\frac{k}{\paramV^2},\quad\E(\ranV-\E{\ranV})^3=\frac{2k}{\paramV^3}.
\end{equation*}
\end{case}

\begin{case}[Pareto]\label{345y45ew}
The random variable $\ranV$ is Pareto if for $\aa>0$ and $\bb>0$ its p.d.f. is
\begin{equation*}
f_{\ranV}(x)=\dfrac{\aa\bb}{(x\bb+1)^{\aa+1}},\quad x>0,
\end{equation*}
and $0$, as $x\leqslant 0$. Plainly, the corresponding c.d.f.
$F_{\ranV}(x)=\int_{0}^{x}f_{\ranV}(z)dz$ is
\begin{equation}\label{werhtjrt}
F_{\ranV}(x)=1-\frac{1}{(x\bb+1)^{\aa}},\quad x>0,
\end{equation}
and $0$, as $x\leqslant 0$.
By direct calculations, it is easy to check that for $\aa>3$ we have
\begin{equation*}
\E{\ranV}=\frac{1}{(\aa-1)\,\bb},\quad\D{\ranV}=\frac{\aa}{(\aa-1)^2(\aa-2)\,\bb^2},
\quad\E(\ranV-\E{\ranV})^3=\frac{2\aa(\aa+1)}{(\aa-1)^3(\aa-2)(\aa-3)\,
\bb^3}.
\end{equation*}
\end{case}

\subsection{Algorithm of simulation}\label{ewrtyurtujtr}

The starting point for the entire simulation process is a pseudo-random number
generator from Uniform [0,1] distribution. We deal with the standard (see
\citeNP{[Knuth 1981]}) linear congruence random number
generator\footnote{Though presumably some built-in pseudo-random number
generators implemented in most standard symbolic computation packages such as
Maple may be in some cases superior to that pseudo-random number generator, we
use it to avoid ``black boxes'' in the description of the algorithm. We bear in
mind that every random number generator has its advances and deficiencies, see
\citeNP{[Hellekalek 1998]}. Quoting from Section 10 of this paper which
discusses criteria for good random number generators, we agree that ``random
number generators are like antibiotics. Every type of generator has its
unwanted side-effects. There is no safe generators. Good random number
generators are characterized by theoretical support, convincing empirical
evidence, and positive practical aspects. They will produce correct results in
many, though in not all, situations.''} based on the equation
\begin{equation*}\label{efdgbf}
x_{n+1}=(\kappa x_{n}+\varrho)\ \mod\ m,
\end{equation*}
where $\kappa=23456789$ is the multiplier, $\varrho=22185$ is the increment,
and $m=2^{32}$ is the modulus. The initial seed $x_{0}$ is selected using a
build-in Maple procedure. Each successive term is transformed into the next.
The pseudo-random terms are in the range from $1$ to $m-1$. To get floating
point numbers between $0$ and $1$, a floating point division by $m$ is done. It
is known that matching of the numbers thus produced to a sample from the
Uniform $[0,1]$ distribution depends heavily on the choice of $\kappa$ and $m$.

Using this pseudo-random number generator from Uniform [0,1] distribution,
pseudo-random numbers from Exponential, Mixture of two Exponentials,
Erlang\footnote{For Erlang with parameters $\paramV>0$ and integer $k$,
simulation may be based on the fact that it is a sum of $k$ i.i.d. Exponential
random variables with parameter $\paramV>0$.}, and Pareto distributions are all
obtained using the method of inverse transforms (see, e.g., \citeNP{[Devroye
1986]}). For instance, pseudo-random number from Mixture of two Exponentials
is $E=F^{-1}_{\ranV}(U)$ with c.d.f. $F_{\ranV}$ given in
\eqref{ewrherhjedtj}, or for $\lambda_1=1$ and $\lambda_2=2$ in explicit
form\footnote{In the general case, for arbitrarily chosen $\lambda_1>0$ and
$\lambda_2>0$, there is no explicit expression and one should solve numerically
the equation $F_{\ranV}(E)=U$ with respect to $E$.}
\begin{equation}\label{wqerthujrykt}
E=-\ln\bigg(\frac{-p+\sqrt{p^2+4q(1-U)}}{2q}\bigg),
\end{equation}
and pseudo-random number from Pareto 
is $P=F^{-1}_{\ranV}(U)$ with
c.d.f. $F_{\ranV}$ given in \eqref{werhtjrt}, or in explicit form
\begin{equation}\label{23456yu7jk5}
P=\frac{1}{b}\bigg(\frac{1}{(1-U)^{1/a}}-1\bigg),
\end{equation}
where $U$ is pseudo-random number from Uniform $[0,1]$ distribution produced by
the generator described above.

To evaluate, using numerical simulation, the probability $\P\{v<\timeR\leqslant
t\mid \T{1}=v\}$ as a function of $c$, while $u$, $t$, $v$ are fixed, we
address the interval $[c_{\min},c_{\max}]$ on the abscissa axis, where
$\cS=1/M\in[c_{\min},c_{\max}]$. We introduce the lattice
\begin{equation*}
\mathcal{C}=\{c_i,\ i=0,1,\dots,n_{\mathcal{C}}\}
\end{equation*}
with the span $\Delta c>0$, i.e. put $c_0=0$ and $c_i=c_{i-1}+\Delta c$,
$i=1,2,\dots,[c_{\max}/\Delta c]+1$.

Starting with $c_0$, we iterate through the nodes of $\mathcal{C}$. Dealing
with the node $c_i$, we simulate the values $\P\{v<\timeR\leqslant t\mid
\T{1}=v\}$ at the points $c_i$, $i=0,1,\dots,n_{\mathcal{C}}$ on the basis of
the definition of this probability. Namely, for each $c_i$ we simulate the
bundle consisting of $N$ trajectories of the process $\homV{s}-c_is$. Then we
pick up the ratio of the trajectories that crossed the level $u$ to the total
number of trajectories $N$ and declare it the value of the probability in
question in the node $c_i$.

\subsection{Approximations and simulation results}\label{ewrtyurtujtr}

\begin{figure}[t]
\includegraphics[scale=0.9]{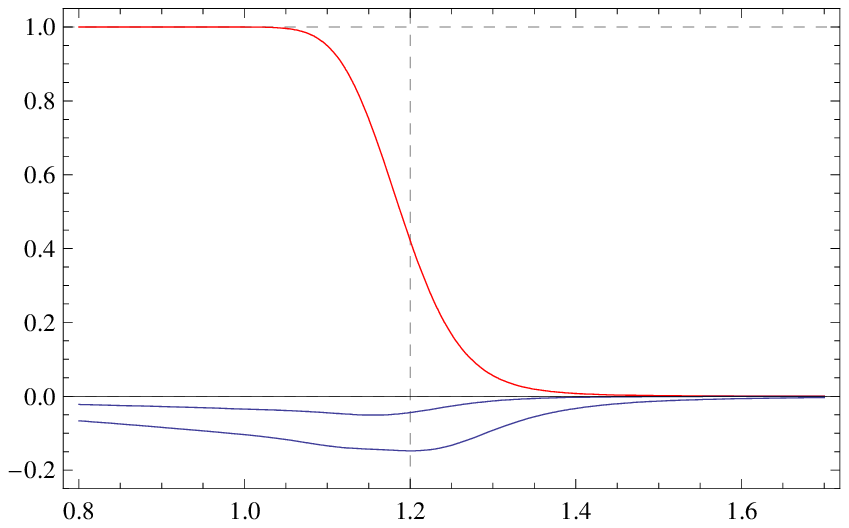}
\includegraphics[scale=0.9]{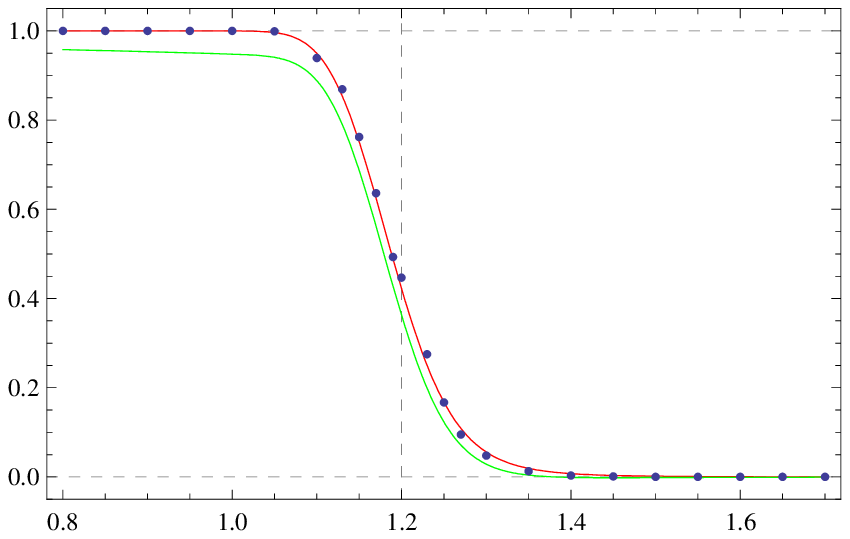}
\caption{\small\!(\emph{Erlang~$Y$--Erlang~$T$})\, Above: graphs ($X$-axis is
$c$) of the functions $\Int{t}{M}(u,c,v)$ (red), $\Int{t}{F}(u,c,v)$,
$\Int{t}{S}(u,c,v)$; below: graphs of the functions $\Int{t}{M}(u,c,v)$ (red),
$\expaN{t}(u,c,v)$ (green) in the case when $Y$ is Erlang with parameters
$\paramY=1$, $m=2$, $T$ is Erlang with parameters $\paramT=1.2$, $k=2$, and
$t=1000$, $v=0$, $u=40$. In this case, $\cS=1/M=1.2$, $D^2=1.39$,
$\KonstF=0.6\,c^{-1}$, $\KonstS=0.3\,c^{-1}$. By dots, shown are the results of
simulation ($\Delta c=0.05$, $N=1000$) according to the algorithm described in
Section~\ref{ewrtyurtujtr}.}\label{wertgreh}
\end{figure}

The following Lemma~\ref{ergwererr} is applied to calculate approximations
$\Int{t}{M}(u,c,v)$ and $\expaN{t}(u,c,v)$ for Erlang $Y$ and $T$, shown in
Fig.~\ref{wertgreh}.

\begin{lemma}[Erlang $Y$--Erlang $T$]\label{ergwererr}
For $Y$ Erlang with parameters $\paramY>0$ and integer $m$, and $T$ Erlang with
parameters $\paramT>0$ and integer $k$, i.e., for
\begin{equation*}
f_{Y}(x)=\frac{\paramY^m x^{m-1}}{\Gamma(m)}\,e^{-x\paramY},\quad
f_{T}(x)=\frac{\paramT^k x^{k-1}}{\Gamma(k)}\,e^{-x\paramT},\quad x>0,
\end{equation*}
and $0$, as $x\leqslant 0$, we have
\begin{equation*}
M=\frac{\paramY k}{\paramT m},\quad D^2=\frac{\paramY
k(k+m)}{m^2\paramT^2},\quad \KonstF=\frac{\paramT m((2+m)k-2m)}{2\paramY
k(k+m)\,c},\quad\KonstS=\frac{\paramT m(k+2m)}{6\paramY k(k+m)\,c}.
\end{equation*}
\end{lemma}

The following corollary is straightforward, if we put $m=1$.

\begin{corollary}[Exponential $Y$--Erlang $T$]\label{srgtewhgewrh}
For $Y$ Exponential with parameter $\paramY>0$, and $T$ Erlang with parameters
$\paramT>0$ and integer $k$, we have
\begin{equation*}
M=\frac{\paramY k}{\paramT},\quad D^2=\frac{\paramY k(k+1)}{\paramT^2},\quad
\KonstF=\frac{\paramT(3k-2)}{2\paramY k(k+1)\,c},
\quad\KonstS=\frac{\paramT(k+2)}{6\paramY k(k+1)\,c}.
\end{equation*}
\end{corollary}

In upper and lower Fig.~\ref{wertgreh}, in the case when $Y$ is Erlang with
parameters $\paramY=1$, $m=2$, $T$ is Erlang with parameters $\paramT=1.2$,
$k=2$, and $t=1000$, $v=0$, $u=40$, visualized are the functions
$\Int{t}{M}(u,c,v)$, $\Int{t}{F}(u,c,v)$, $\Int{t}{S}(u,c,v)$ which, together
with the factors $\KonstF=0.6\,c^{-1}$, $\KonstS=0.3\,c^{-1}$, are involved
(see \eqref{wsdrthyjkr}) in the construction of $\expaN{t}(u,c,v)$. Visualized
are also the approximations $\Int{t}{M}(u,c,v)$ and $\expaN{t}(u,c,v)$,
compared with the results of simulation shown by dots. The simulation algorithm
is described in Section~\ref{ewrtyurtujtr} ($\Delta c=0.05$ and more frequent
in the flexure region, $N=1000$). Simulation demonstrates that in this test
case $\Int{t}{M}(u,c,v)$ looks preferable for $c<\cS$, though both
$\Int{t}{M}(u,c,v)$ and $\expaN{t}(u,c,v)$ look equally accurate for $c>\cS$.

The following Lemma~\ref{wertyuhrtuj} is applied to calculate approximations
$\Int{t}{M}(u,c,v)$ and $\expaN{t}(u,c,v)$ in the case of Pareto $Y$ and
Mixture of two Exponentials $T$, shown in Fig.~\ref{sfdgrthjn}.

\begin{figure}[t]
\includegraphics[scale=0.9]{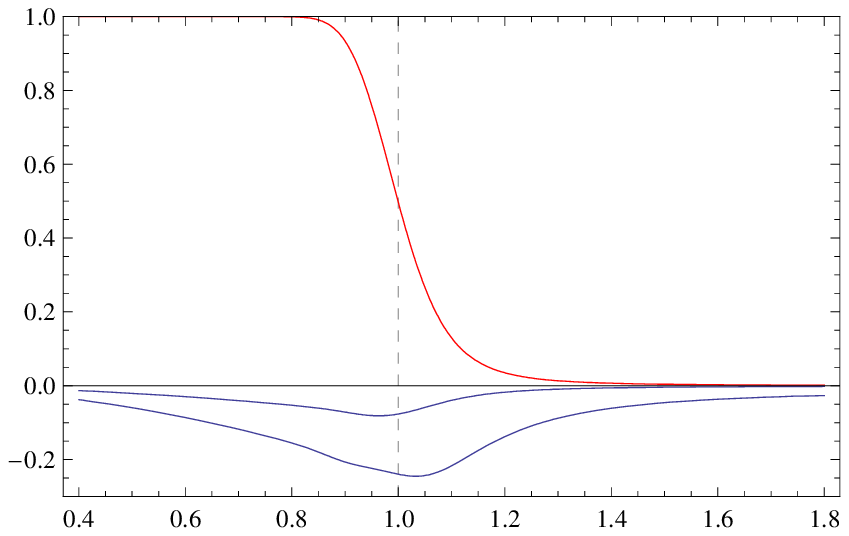}
\includegraphics[scale=0.9]{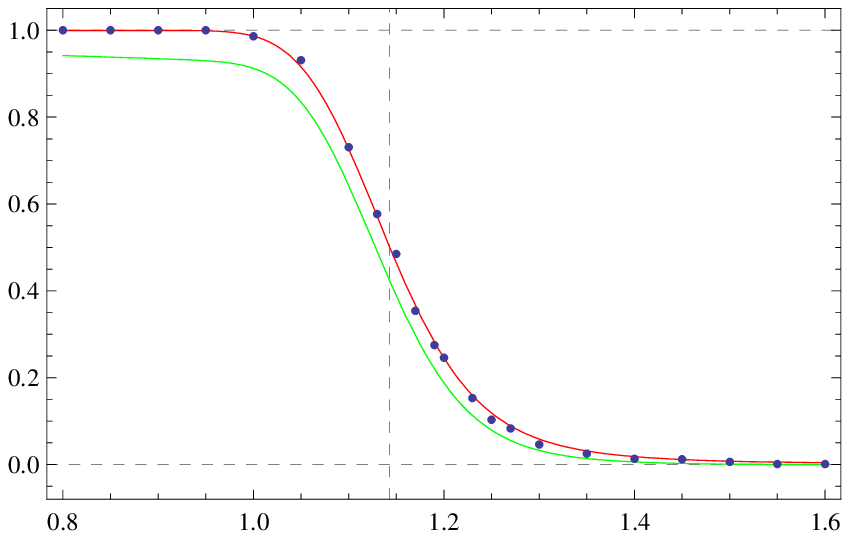} \caption{\small\!(\emph{Pareto~$Y$--Mixture
of two Exponentials~$T$})\, Above: graphs ($X$-axis is $c$) of the functions
$\Int{t}{M}(u,c,v)$ (red), $\Int{t}{F}(u,c,v)$, $\Int{t}{S}(u,c,v)$; below:
graphs of the functions $\Int{t}{M}(u,c,v)$ (red), $\expaN{t}(u,c,v)$ (green)
in the case when $Y$ is Pareto with parameters $\a=4.0$, $\b=0.35$, $T$ is
Mixture of two Exponentials with parameters $\lambda_1=1$, $\lambda_2=2$,
$p=2/3$, and $t=1000$, $v=0$, $u=40$. In this case, $\cS=1/M=1.143$,
$D^2=2.304$, $\KonstF=1.04\,c^{-1}$, $\KonstS=0.076\,c^{-1}$. By dots, shown
are the results of simulation ($\Delta c=0.05$, $N=1000$) according to the
algorithm described in Section~\ref{ewrtyurtujtr}.}\label{sfdgrthjn}
\end{figure}

\begin{lemma}[Pareto~$Y$--Mixture of two
Exponentials~$T$]\label{wertyuhrtuj} For $Y$ Pareto with parameters $\a>3$,
$\b>0$, and $T$ Mixture of two Exponentials with parameters $\lambda_1>0$,
$\lambda_2>0$, and $0<p<1$ {\rm (}$q=1-p${\rm )}, i.e., for
\begin{equation*}
f_{Y}(x)=\dfrac{\a\b}{(x\b+1)^{\a+1}},\quad
f_{T}(x)=\lambda_1pe^{-\lambda_1x}+\lambda_2qe^{-\lambda_2 x},\quad x>0,
\end{equation*}
and $0$, as $x\leqslant 0$, we have
\begin{equation*}
\begin{gathered}
M=(\a-1)\b\bigg(\frac{p}{\lambda_1}+\frac{q}{\lambda_2}\bigg),
\\[4pt]
D^2=(\a-1)\b\bigg(\frac{\a}{\a-2}\Big(\frac{p}{\lambda_1}+\frac{q}{\lambda_2}\Big)^2+\frac{\lambda_2^2p+\lambda_1^2q
+(\lambda_1-\lambda_2)^2pq}{\lambda_1^2\lambda_2^2}\bigg),
\end{gathered}
\end{equation*}
\begin{multline*}
\KonstF=\frac{(\a-2)\lambda_1\lambda_2}{2(\a-3)(\a-1)
\b\,\big(4\lambda_1\lambda_2pq+2\lambda_1^2q(\a-(1+p))+2\lambda_2^2p(\a-(1+q))\big)^2\,c}
\\[4pt]
\times\Big(\lambda_2^3\big((-12+9\a+\a^2)p^3-12(6-5\a+\a^2) pq-2(6-5\a+\a^2)q^3
\\
+(6-5\a+\a^2)p^2(1+q)\big)+\lambda_1^3q^2\big(6(1+p)-\a(6+4p)+\a^2
(1+p+3q)\big)
\\
+\lambda_1\lambda_2^2pq\big(6-12p+\a(-5+7p-35q)+42q+\a^2(1+7p+7q)\big)
\\
+\lambda_1^2\lambda_2pq\big(6(1+7p-2q)+\a(-5-35p+7q)+\a^2(1+7p+7q)\big)\Big),
\end{multline*}
\begin{multline*}
\KonstS=\frac{(\a-2)^2\lambda_1\lambda_2^4}{6(\a-1)\b(4\lambda_1\lambda_2pq+
2\lambda_1^2q(\a-(1+p))+4\lambda_2^2p(\a-(1+q)))^2\,c}
\\[4pt]
\times\bigg(-\frac{6\lambda_1^2p^2q}{\lambda_2^2}-\frac{6\lambda_1pq^2}{\lambda_2}+2p(p^2+3q)-\frac{2
\a(1+\a)(\lambda_2p+\lambda_1q)^3}{(6-5\a+\a^2)\lambda_2^3}+2q(3p+q^2)
\\[-2pt]
+\frac{3\a(\lambda_2p+\lambda_1q)}{(\a-2)^2\lambda_2^3}\Big(4\lambda_1\lambda_2pq+2\lambda_1^2q
\big(\a-(1+p)\big)+2\lambda_2^2p\big(\a-(1+q)\big)\Big)\bigg).
\end{multline*}
\end{lemma}

In Fig.~\ref{sfdgrthjn}, visualization is done in the case when $Y$ is Pareto
with parameters $\a=4.0$, $\b=0.35$, and $T$ is Mixture of two Exponentials
with parameters $\lambda_1=1$, $\lambda_2=2$, $p=2/3$, and $t=1000$, $v=0$,
$u=40$. Simulation is done according to the algorithm described in
Section~\ref{ewrtyurtujtr} ($\Delta c=0.05$ and more frequent in the flexure
region, $N=1000$), using equation \eqref{23456yu7jk5} for Pareto $Y$ and
equation \eqref{wqerthujrykt} for Mixture of two Exponentials $T$ with
$\lambda_1=1$ and $\lambda_2=2$.

The following Lemma~\ref{wretgrhre} is applied to calculate approximations
$\Int{t}{M}(u,c,v)$ and $\expaN{t}(u,c,v)$ for Pareto $Y$ and Erlang $T$, shown
in Fig.~\ref{srgeherjhd}.

\begin{figure}[t]
\includegraphics[scale=0.9]{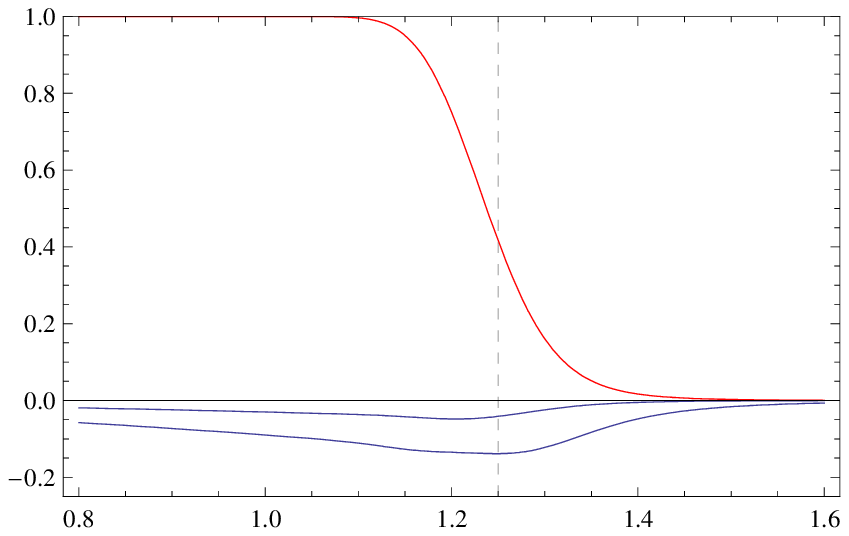}
\includegraphics[scale=0.9]{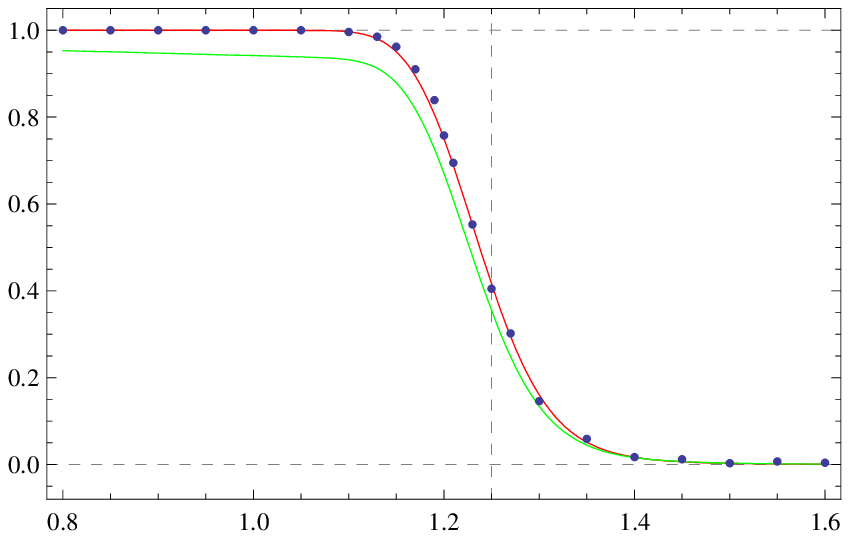}
\caption{\small\!(\emph{Pareto~$Y$--Erlang~$T$})\, Above: graphs ($X$-axis is
$c$) of the functions $\Int{t}{M}(u,c,v)$ (red), $\Int{t}{F}(u,c,v)$,
$\Int{t}{S}(u,c,v)$; below: graphs of the functions $\Int{t}{M}(u,c,v)$ (red),
$\expaN{t}(u,c,v)$ (green) in the case when $Y$ is Pareto with parameters
$\a=4.0$, $\b=0.4$, $T$ is Erlang with parameters $\paramT=6.0$, $k=4$, and
$t=1000$, $v=0$, $u=40$. In this case, $\cS=1/M=1.25$, $D^2=1.2$,
$\KonstF=2.73\,c^{-1}$, $\KonstS=-0.26\,c^{-1}$. By dots, shown are the results
of simulation ($\Delta c=0.05$, $N=1000$) according to the algorithm described
in Section~\ref{ewrtyurtujtr}.}\label{srgeherjhd}
\end{figure}

\begin{lemma}[Pareto $Y$--Erlang $T$]\label{wretgrhre}
For $Y$ Pareto with parameters $\a>3$, $\b>0$, and $T$ Erlang with parameters
$\paramT>0$ and integer $k$, i.e., for
\begin{equation*}
f_{Y}(x)=\dfrac{\a\b}{(x\b+1)^{\a+1}},\quad f_{T}(x)=\frac{\paramT^k
x^{k-1}}{\Gamma(k)}e^{-x\paramT},\quad x>0,
\end{equation*}
and $0$, as $x\leqslant 0$, we have
\begin{equation*}
M=\frac{k(\a-1)\b}{\paramT},\quad
D^2=\frac{k(\a-1)\b}{\paramT^2}\bigg(1+\frac{k\a}{\a-2}\bigg),
\end{equation*}
\begin{equation*}
\begin{gathered}
\KonstF=(\a-2)\paramT\,\frac{\a^2(-2+k+3k^2)-\a(-10+5k+k^2)+6(-2+k)}{2(\a-1)(\a-3)
\b k(-2+\a+\a k)^2\,c},
\\[4pt]
\KonstS=\paramT\,\frac{\a^3(2+3k+k^2)-\a^2(14+15k+7k^2)+2\a(16+9k+2k^2)-24}{6(\a-1)(\a-3)
\b k (-2+\a+\a k)^2\,c}.
\end{gathered}
\end{equation*}
\end{lemma}

In Fig.~\ref{srgeherjhd}, visualized is the case when $Y$ is Pareto with
parameters $\a=4.0$, $\b=0.4$, $T$ is Erlang with parameters $\paramT=6.0$,
$k=4$, and $t=1000$, $v=0$, $u=40$. Simulation is done according to the
algorithm described in Section~\ref{ewrtyurtujtr} ($\Delta c=0.05$ and more
frequent in the flexure region, $N=1000$), using equation \eqref{23456yu7jk5}
and bearing in mind that this Erlang $T$ is a sum of four i.i.d. Exponential
summands with parameter $\paramT=6.0$.

The following Lemma~\ref{wqertgrjhtjmn} is applied to calculate approximations
$\Int{t}{M}(u,c,v)$ and $\expaN{t}(u,c,v)$ for Pareto $Y$ and $T$, shown in
Fig.~\ref{swrgtehbrhr}.

\begin{figure}[t]
\includegraphics[scale=0.9]{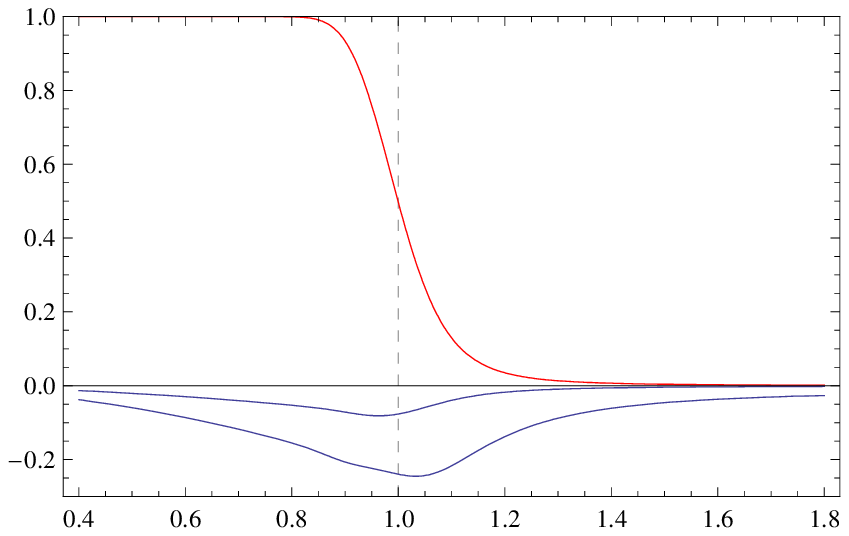}
\includegraphics[scale=0.9]{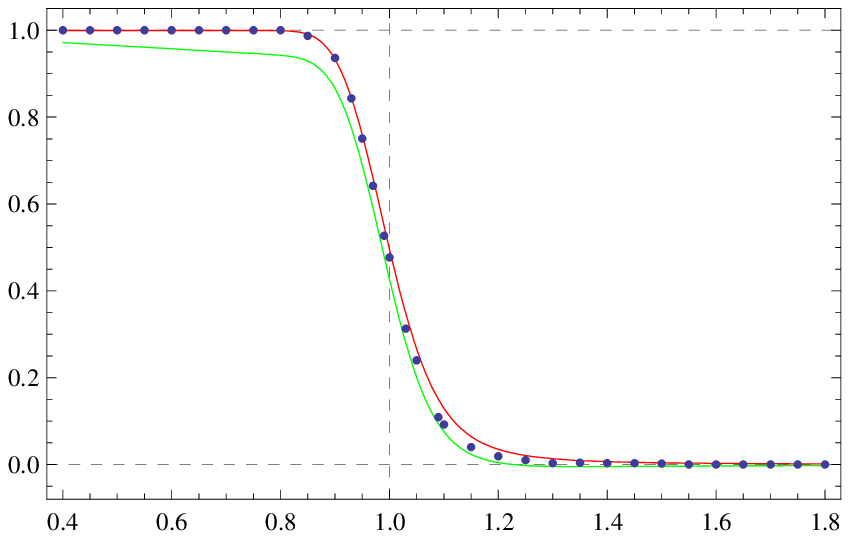}
\caption{\small\!(\emph{Pareto~$Y$--Pareto~$T$})\, Above: graphs ($X$-axis is
$c$) of the functions $\Int{t}{M}(u,c,v)$ (red), $\Int{t}{F}(u,c,v)$,
$\Int{t}{S}(u,c,v)$; below: graphs of the functions $\Int{t}{M}(u,c,v)$ (red),
$\expaN{t}(u,c,v)$ (green) in the case when $T$ is Pareto with parameters
$\d=4.0$, $\g=0.4$, $Y$ is Pareto with parameters $\a=4.0$, $\b=0.4$, and
$t=1000$, $v=0$, $u=40$. In this case, $\cS=1/M=1$, $D^2=3.333$,
$\KonstF=0.125\,c^{-1}$, $\KonstS =0.25\,c^{-1}$. By dots, shown are the
results of simulation ($\Delta c=0.05$, $N=1000$) according to the algorithm
described in Section~\ref{ewrtyurtujtr}.}\label{swrgtehbrhr}
\end{figure}


\begin{lemma}[Pareto $Y$--Pareto $T$]\label{wqertgrjhtjmn}
For $Y$ Pareto with parameters $\a>3$, $\b>0$, and $T$ Pareto with parameters
$\d>3$, $\g>0$, i.e., for
\begin{equation*}
f_{Y}(x)=\dfrac{\a\b}{(x\b+1)^{\a+1}},\quad
f_{T}(x)=\dfrac{\d\g}{(x\g+1)^{\d+1}},\quad x>0,
\end{equation*}
and $0$, as $x\leqslant 0$, we have
\begin{equation*}
M=\frac{(\a-1)\b}{(\d-1)\g},\quad D^2=
\bigg(\frac{\a}{\a-2}+\frac{\d}{\d-2}\bigg)\frac{(\a-1)\b}{(\d-1)^2\g^2},
\end{equation*}
\begin{multline*}
\KonstF=\frac{(\d-2)(\a-2)\g}{4(\d-3)(\a-3)\b}
\\[4pt]
\times\frac{\a^2(9-10\d+\d^2)+\a(-3+15\d+2\d^2)-3\d(5+\d)}{(\a-1)^3(\d-1)\,c},
\end{multline*}
\begin{multline*}
\KonstS=\frac{(\d-2)\g}{4(\d-3)(\a-3)\b}
\\[4pt]
\times\frac{\a^3(\d-1)^2-4\d(1+\d)+\a^2(-7+11\d-6\d^2)+\a(4-7\d
+9\d^2)}{(\a-1)^3(\d-1)\,c}.
\end{multline*}
\end{lemma}

In Fig.~\ref{swrgtehbrhr}, visualized are the functions $\Int{t}{M}(u,c,v)$,
$\Int{t}{F}(u,c,v)$, $\Int{t}{S}(u,c,v)$ which together with the factors
$\KonstF=0.125\,c^{-1}$, $\KonstS =0.25\,c^{-1}$ are involved (see
\eqref{wsdrthyjkr}) in the construction of $\expaN{t}(u,c,v)$, and the
approximations $\Int{t}{M}(u,c,v)$, $\expaN{t}(u,c,v)$ in the case when $T$ is
Pareto with parameters $\d=4.0$, $\g=0.4$, $Y$ is Pareto with parameters
$\a=4.0$, $\b=0.4$, and $t=1000$, $v=0$, $u=40$. By dots, shown are the results
of simulation ($\Delta c=0.05$ and more frequent in the flexure region,
$N=1000$) carried out according to the algorithm described in
Section~\ref{ewrtyurtujtr}, using equation \eqref{23456yu7jk5}.

The comparison of $\Int{t}{M}(u,c,v)$ and $\expaN{t}(u,c,v)$ with the results
of simulation shows that in all test cases of
Figs.~\ref{wertgreh}--\ref{swrgtehbrhr} both these approximations are
surprisingly accurate, in particular in the vicinity of the point $\cS=1/M$,
though $u=40$ chosen in Figs.~\ref{wertgreh}--\ref{swrgtehbrhr} is rather
moderate.

\section{Conclusions}\label{srdtyrujhrt}

In the case of Exponential $T$ and $Y$, we have compared (a) numerical results
yielded by approximations of Theorems~\ref{srdthjrf} and \ref{erty5uyh4tj}, (b)
numerical results derived by means of exact formula of Theorem~\ref{dthyjrt},
and (c) numerical results yielded by simulation. The availability of the exact
formula in this case allows us, inter alia, to be confident in the error-free
operating of the computer simulation program.

For $T$ and $Y$ non-Exponential, when exact formulas like in
Theorems~\ref{dthyjrt} do not exist or are excessively cumbersome (see, e.g.,
\citeNP{[Borovkov Dickson 2008]}), in our hands remains only simulation
technique. We use it for getting numerical benchmarks needed to verify and
evaluate performance of the approximations $\Int{t}{M}(u,c,v)$ and
$\expaN{t}(u,c,v)$ put forth in Theorems~\ref{srdthjrf} and \ref{erty5uyh4tj}.
The comparison of $\Int{t}{M}(u,c,v)$ and $\expaN{t}(u,c,v)$ and simulation
results done in Section~\ref{ewrtherherhr} indicates a good quality of these
approximations.

It is noteworthy that the approximations $\Int{t}{M}(u,c,v)$ and
$\expaN{t}(u,c,v)$, unlike, e.g., the famous Cram{\'e}r-type approximation,
hold true not only for the distributions of $Y$ with exponentially decreasing
tail, but also for heavy-tailed $Y$, so we include Pareto distribution in our
numerical analysis. Examining simulation results given in
Figs.~\ref{sdrtfrjmty}, \ref{ertgerhrh}, and \ref{wertgreh}--\ref{swrgtehbrhr},
we see that the approximations $\Int{t}{M}(u,c,v)$ and $\expaN{t}(u,c,v)$ are
very satisfactory for both light-tailed and heavy-tailed $T$ and $Y$. This
examining confirms that the approximations $\Int{t}{M}(u,c,v)$ and
$\expaN{t}(u,c,v)$ are satisfactory uniformly on $c$, including vicinity of the
critical point $\cS=1/M$, in all cases considered. Generally, the accuracy of
approximations $\Int{t}{M}(u,c,v)$ and $\expaN{t}(u,c,v)$ is visibly better
than, e.g., of the Cram{\'e}r-type approximation all over the range of $c$,
including outside the vicinity of $\cS$; more detailed discussion of the
deficiencies of Cram{\'e}r-type approximation, from which $\Int{t}{M}(u,c,v)$
and $\expaN{t}(u,c,v)$ are free, may be seen in \citeNP{[Malinovskii Kosova
2014]}.


\begin{thebibliography}{99}
\baselineskip=12pt

{\normalsize


\bibitem[Abramowitz and Stegun (1972)]{[Abramowitz Stegun 1972]}
\hskip -20pt\textsc{Abramowitz,\,M., and Stegun,\,I.A.} (1972) \emph{Handbook
of Mathematical Functions}, 10-th ed., Dover, New York.

\bibitem[Borovkov and Dickson (2008)]{[Borovkov Dickson 2008]}
\hskip -20pt\textsc{Borovkov,\,K., and Dickson,\,D.C.M.} (2008) On the ruin
time distribution for a Sparre Andersen process with exponential claim sizes,
{\IME}, Vol.~42, 1104--1108.

\bibitem[Devroye (1986)]{[Devroye 1986]}
\hskip -20pt\textsc{Devroye, L.} (1986) \emph{Non-uniform random variate
generation}. Springer-Verlag, New York.

\bibitem[Hellekalek (1998)]{[Hellekalek 1998]}
\hskip -20pt\textsc{Hellekalek, P.} (1998) Good random number generators are
(not so) easy to find, \emph{Ma\-the\-matics and Computers in Simulation},
vol.~46, 485--505.

\bibitem[Knuth (1981)]{[Knuth 1981]}
\hskip -20pt\textsc{Knuth, D.E.} (1981) \emph{The Art of Computer Programming}, Vol.2,
Seminumerical Algorithms, 2nd ed., Addison Wesley.

\bibitem[Ma\-li\-nov\-skii (2017a)]{[Malinovskii 2017a]}
\hskip -20pt\textsc{Malinovskii,\,V.K.} (2017a) On the time of first level
crossing and inverse Gaussian distribution. Submitted.

\bibitem[Ma\-li\-nov\-skii (2017b)]{[Malinovskii 2017b]}
\hskip -20pt\textsc{Malinovskii,\,V.K.} (2017b) Generalized inverse Gaussian
distributions and the time of first level crossing. Submitted.

\bibitem[Malinovskii and Kosova (2014)]{[Malinovskii Kosova 2014]}
\hskip -20pt\textsc{Malinovskii,\,V.K., and Kosova,\,K.O.} (2014) Simulation
analysis of ruin capital in Sparre An\-der\-sen's model of risk, {\IME},
Vol.~59, 184--193.

\bibitem[Watson (1945)]{[Watson 1945]}
\hskip -20pt\textsc{Watson,\,G.N.} (1945) \emph{A Treatise on the Theory of
Bessel Functions.} Cambridge University Press, Cambridge.

\bibitem[Whittaker and Watson (1963)]{[Whittaker Watson 1963]}
\hskip -20pt\textsc{Whittaker,\,E.T., and Watson,\,G.N.} (1963) \emph{A Course
of Modern Analysis.} 4-th ed., Cambridge University Press, Cambridge.

}\end{thebibliography}
\end{document}